\documentclass[a4paper,10pt,hidelinks]{article}

\usepackage{amsmath,amscd,amssymb,amsthm,amssymb}

\usepackage{cite}
\usepackage{bibgerm}

\usepackage[super]{nth}

\usepackage[british, USenglish]{babel}
\newtheorem{thm}{Theorem}[section]
\newtheorem{prop}[thm]{Proposition}

\newtheorem{cor}[thm]{Corollary}
\theoremstyle{definition}
\newtheorem{rem}[thm]{Remark}
\newtheorem{dfn}[thm]{Definition}
\newtheorem{exmpl}[thm]{Example}

\newcommand{\e}{\varepsilon}
\renewcommand{\a}{\alpha}

\newcommand{\s}{\sigma}
\renewcommand{\k}{\kappa}

\newcommand{\fg}{\mathfrak g}

\newcommand{\fh}{\mathfrak h}
\newcommand{\fk}{\mathfrak k}

\newcommand{\fsl}{\mathfrak{sl}}
\newcommand{\fso}{\mathfrak{so}}

\newcommand{\fe}{\mathfrak e}

\newcommand{\Diff}{{\rm{\bf Diff}}}
\newcommand{\Eg}{{\rm {\bf E}}}

\newcommand{\Hg}{{\rm {\bf H}}}
\newcommand{\SOg}{{\rm{\bf SO}}}
\newcommand{\GLg}{{\rm{\bf GL}}}
\newcommand{\Derg}{{\rm{\bf Der}}}
\newcommand{\Aut}{{\rm {\bf Aut}}}

\newcommand{\End}{{\rm {\bf End}}}

\newcommand{\Ug}{{\rm {\bf U}}}
\newcommand{\SUg}{{\rm {\bf SU}}}
\newcommand{\SLg}{{\rm {\bf SL}}}

\newcommand{\Nil}{{\rm {\bf Nil}}}

\newcommand{\NN}{\mathbb N}

\newcommand{\ZZ}{\mathbb Z}
\newcommand{\RR}{\mathbb R}
\newcommand{\CC}{\mathbb C}

\newcommand{\ad}{{\rm ad}}

\newcommand{\tr}{{\rm tr\,}}

\newcommand{\id}{{\rm id}}

\renewcommand{\div}{{\rm div}}

\newcommand{\Fix}{{\rm Fix}}
\newcommand{\de}{{\rm d\hspace{0.2mm}}}
\newcommand{\lr}{\longrightarrow}
\newcommand{\De}{{\rm D\hspace{0.2mm}}}

\newcommand{\pr}{{\rm pr}}

\title{A classification of Thurston geometries without compact quotients}
\author{Panagiotis Konstantis\footnote{Phillips Universit\"at Marburg}\quad\quad\quad Frank Loose\footnote{Eberhard Karls Universit\"at T\"ubingen}}
\date{}

\begin{document}

\maketitle

\begin{abstract}
We classify pairs $(M,G)$ where $M$ is a $3$--dimensional simply connected smooth manifold and $G$ a Lie group acting on $M$ transitively, effectively with compact isotropy group. 
\end{abstract}

\section{Introduction}

{\bf Background.}\;\;\;Over the last few decades the study of closed $3$-manifolds was a very active topic in geometry and topology, which culminated in the remarkable proof of \emph{Thurston's geometrization conjecture} by G. Perelman. It was known since  \cite{Kneser1929} and \cite{Milnor1962} that every closed $3$-manifold $M$ can be decomposed uniquely into prime $3$-manifolds. On the other hand W. Thurston conjectured in \cite{Thurston1982} that if a prime manifold is cut along some tori, which is also known as the Jaco-Shalen-Johannson Torus Decomposition, the interior admits a unique geometric structure of finite volume.  This means that the universal cover of the interior is a \emph{model geometry} also known as \emph{Thurston geometry}. These are defined as pairs $(M,G)$ where $M$ is a simply connected, smooth, $3$-dimensional manifold and $G$ a Lie group of diffeomorphisms of $M$, such that $G$ acts transitively and with compact stabilizers. Moreover, $(M,G)$ should admit compact quotients (i.e. 
there 
is a discrete subgroup $\Gamma \subset G$ such that $M/\Gamma$ is a smooth compact manifold) and $G$ should be \emph{maximal} in the sense that $G$ is not contained in a bigger Lie group of diffeomorphisms with same properties.


The list of Thurston geometries may be found, e.g. in \cite{Thurston1997} or in \cite{Scott1983}, however it is difficult to find a stringent proof of this classification among the literature. One aim of this article is to present a complete proof of this classification following the arguments in \cite{Thurston1997}.

Another aim is to study model geometries without assuming the existence of compact quotients nor requiring the maximality of $G$. We call such a pair $(M,G)$ simply a \emph{geometry}. These objects may be interesting for applications in physics, since they may be used to model the spatial parts of homogeneous space-times. For example in \cite{Konstantis2013} they were used to study non-isotropic, homogeneous cosmological models with positive cosmological constant. Finally it should be mentioned that in \cite[p.474]{Scott1983} there is a remark wherein R. S. Kulkarni carried out such classification but remained unpublished. 

\bigskip
\noindent
{\bf Overview of the classification.}\;\;\;As we already pointed out the classification will follow the outline of the proof in \cite{Thurston1997} but without using the assumption that $(M,G)$ has compact quotients. For the sake of convenience we assume that $G$ is connected. An easy argument shows that the stabilizer $K$ of $(M,G)$ in a point has to be connected and since $K$ is compact, it has to be a subgroup of $\SOg(3)$. It follows that the possible dimensions of $K$ are $0$, $1$ and $3$ since the Lie algebra of $\SOg(3)$ clearly has no two-dimensional Lie subalgebra. So we divide the geometries according to the dimension of their isotropy groups:

\bigskip

{$\bullet$ ${\dim K=3}$}\;: We call this type an \emph{isotropic geometry}. Clearly we have $K=\SOg(3)$ and $M$ admits a metric of constant sectional curvature. Hence $M$ is a space form and $G$ is the identity component of the full isometry group of the standard metric of this space form, compare Theorem 3.8.4 (a) in \cite{Thurston1997}.

\bigskip

{$\bullet$ ${\dim K=1}$}\;: We call this type an \emph{axially symmetric geometry}. The isotropy group is isomorphic to $\SOg(2)$ and it acts on the tangent space of a fixed point as a rotation around a unique line. This implies the existence of a $G$--invariant vector field $X$ which in turn gives a foliation into geodesics of $M$ through its flow lines. Here we give a precise argument why the space of leafs $N$ inherits a differentiable structure from $M$ (compare \cite[p. 183]{Thurston1997}). The crucial fact is that the flow acts properly on $M$ which is proven in Corollary \ref{CProperAction}. 

Moreover the quotient map $\pi \colon M \to N$ is a fiber bundle with $\RR$ or $\SOg(2)$ as fiber group. We explain in section \ref{STheGInvariantConnectionOnMtoN} that the plane field orthogonal to $X$ with respect to any $G$-invariant metric is unique and defines a connection for this bundle, compare \cite[p. 183]{Thurston1997} and Corollary \ref{CGinvariantConnection}. Since this connection is $G$-invariant its curvature is either globally zero or not. Furthermore, from now on the classification in \cite{Thurston1997} makes heavily use of the existence of compact quotients. A key point is here that if $(M,G)$ admits compact quotients, then $\div X$, which is the divergence of $X$ with respect to a $G$--invariant metric, is zero, cf. \cite[p. 182]{Thurston1997}. But then it follows from Proposition \ref{PDivergenceOfX} that $X$ has to be a Killing field for a $G$-invariant metric which means that $N$ admits a metric (of constant curvature) such that $\pi \colon M \to N$ is a Riemannian submersion. 

In section \ref{SFlatConnection} we start to classify the geometries which allow such a flat connection. In \cite{Thurston1997} there is no argument to prove this step. In our approach the key fact for this case is Proposition \ref{PExtensionInFlatCase} where $G$ fits into an exact sequence 
\[
1 \longrightarrow G' \longrightarrow G \longrightarrow \RR \to 1
\]
where $G'\in\{\Eg_0(2),\SOg(3),\SOg^+(2,1)\}$. Now if $(M,G)$ admits compact quotients this sequence splits through the flow of $X$, therefore $G=G' \times \RR$ and we obtain the geometries of $(b_1)$ in \cite[p. 183]{Thurston1997} (compare Remark \ref{RCompareToThurston}). Otherwise we obtain an additional geometry, cf. Example \ref{EFlatGeometry}. A complete list and proof may be found in Theorem \ref{TClassificationFlatGeometries}.
  
The classification of axially symmetric geometries with non-flat $G$-invariant connection is stated within section \ref{SnonFlatConnection}. We prove with Corollary \ref{CXKillingField}, even though $(M,G)$ does not admit compact quotients, that $X$ is a Killing field. This means basically that we do not obtain a new geometry here. Hence this case does not depend on the existence of compact quotients. Nonetheless there is no precise argument for this step in \cite{Thurston1997} or in \cite{Scott1983}. Therefore we use our Lie group theoretical approach here. Basically we try to obtain the Lie algebra of $G$ by means of extensions of Lie groups and Lie algebras. Afterwards we deduce the Lie group $G$ as well as the action of $G$ on $M$. This is all done in the Propositions \ref{PGeometryS3}, \ref{PGeometrySL} and \ref{PGeometryNil}. The final list is stated in Theorem \ref{TNonFlatGeometries}.

\bigskip

{$\bullet$ ${\dim K=0}$}\;: Here $M$ is the Lie group $G$. Since we assumed $M$ to be simply connected, the Lie group $G$ is uniquely determined by its Lie algebra $\fg$. This was done by many authors, compare \cite{Bianchi1898}, \cite{Bianchi2001}, \cite{Glas2013}, \cite{Konstantis2013}. The $3$-dimensional Lie groups which are Thurston geometries are the \emph{unimodular} Lie groups, compare \cite[p.99]{Milnor1976}. Also in \cite{Milnor1976} one finds a complete classification of unimodular Lie groups. For a more detailed discussion see section \ref{SLieGroups}.

\bigskip
\noindent
{\bf Notations and basic definitions.}\;\;\; A \emph{geometry} is a pair $(M,G)$ where $M$ is a simply connected, smooth $3$-manifold and $G$ is a Lie group of diffeomorphisms acting on $M$ transitively with compact stabilizers. We assumed $G$ to be connected, but with a little more effort the classification can be done without this assumption. Using the long exact homotopy sequence, this implies that the stabilizer $K$ has to be connected, since $M$ is simply connected.

Actions of $\theta \colon G \times X \to X$ on smooth manifolds $X$ will be abbreviated by $\theta(g,x)=:g.x$. Note that $G$ is a subgroup of $\Diff(M)$, hence every element $g \in G$ is viewed as a diffeomorphism on $M$.

For this article we fix a point $m_0 \in M$ and let $K$ be its isotropy group in $G$. We denote with $\rho \colon  K \to \SOg(3)$ its faithful isotropy representation. Hence the dimension of $K$ is either $0$, $1$ or $3$, since the Lie algebra of $\SOg(3)$ has no $2$-dimensional Lie subalgebra.  Therefore the isomorphism class of $K$ is the trivial group, $\SOg(2)$ or $\SOg(3)$ respectively. 

We say two geometries $(M,G)$ and $(M',G')$ are \emph{isomorphic} or \emph{equivariant diffeomorphic} if there is a Lie group isomorphism $\Phi \colon G \to G'$ and a diffeomorphism $\varphi \colon M \to M'$ such that $\varphi$ is $\Phi$--equivariant. 

Some groups will be of great importance in this article. We denote be $\Eg_0(n)$ the connected subgroup of the full group of motions in euclidean space, i.e. $\Eg_0(n)=\RR^n \rtimes \SOg(n)$, where $\SOg(n)$ acts on $\RR^n$ by its standard representation. Furthermore let $\SOg^+(n,1)$ be the connected component of the isometry group of hyperbolic space $D^n$.

At the beginning it was mentioned that the classification of isotropic geometries is very easy. The resulting geometries are given by the pairs
\[
 (\RR^3, \Eg_0(3)), \quad (S^3, \SOg(4)), \quad (D^3,\SOg^+(3,1))
\]
with the standard actions as isometry groups. 

\bigskip
{\bf Acknowledgments.}\;\;\; The first author would like to thank the \emph{Carl Zeiss Stiftung} for financial support and \emph{Ilka Agricola} for helpful comments.

\section{Axially Symmetric Geometries}

The classification of axially symmetric geometries is much harder than the classification of isotropic geometries. Here the crucial fact is that every non-trivial $\SOg(2)$-representation on a $3$-dimensional vector space has a unique decomposition into two irreducible subspaces. Even more, $\SOg(2)$ is acting trivially on the $1$-dimensional subspace. Since the isotropy groups of an axially symmetric geometry are isomorphic to $\SOg(2)$, this induces a $G$--invariant vector field $X$ on $M$ whose flow determines a foliation of $M$ into geodesics. 

In particular we will see that the space of leafs is a smooth manifold such that $M$ is the total space of a principal bundle over a surface. Now the $2$-dimensional irreducible subspace creates a distribution of planes on $M$ which turns out to be a $G$--invariant connection on $M$. The classification will be then divided into two cases, namely if the connection is flat or not.

\subsection{Foliation of $M$ into geodesics}

We start to study the isotropy representation and the properties of a $G$--invariant vector field which comes from that representation. The flow of the vector field will give us a foliation of $M$.

Since $\dim K=1$ we obtain $K \cong \SOg(2)$ and since the isotropy representation $\rho$ is faithful we formulate without proof the crucial 

\begin{prop}\label{PRepresentationSO(2)} 
Let $\rho \colon \SOg(2) \to V$ be a faithful representation on a $3$--dimensional vector space $V$. Then there is a unique $\SOg(2)$--invariant decomposition $V=L\oplus W$ into irreducible subspaces where $\dim L=1$ and $\dim W=2$ such that $\rho$ fixes every point in $L$ and $W$ is the orthogonal complement to $L$ for every $\SOg(2)$--invariant metric.
\end{prop}

It will be of great importance to decompose the $G$--invariant bilinear forms on $M$ with respect to the representation $\rho$. Therefore it is sufficient to know how the $K$--invariant endomorphisms on $(V,\rho)$ decompose. We denote by $\End_\rho(V)$ the algebra of equivariant endomorphisms on $(V,\rho)$.

\begin{prop}\label{PKInvariantEndos}
Suppose $\rho \colon \SOg(2) \to \GLg(V)$ is a faithful representation on a $3$--dimensional vector space $V$. If $f \in \End_\rho(V)$ then it is given by
\[
  \left(
  \begin{matrix}
  \a & 0 \\
  0 & \lambda \cdot k|W
  \end{matrix}
  \right)
\]
with respect to the decomposition $V=L\oplus W$ of Proposition \ref{PRepresentationSO(2)} where $\a \in \End(\RR)\cong \RR$, $\lambda \geq 0$ and $k \in \SOg(2) \subseteq \GLg(V)$.
\end{prop}

\begin{proof}
We abbreviate $K=\SOg(2)$. Choose a $K$--invariant scalar product $\sigma$ on $V$. For $f \in \End(V)$ we denote by $f^*$ the adjoint endomorphism of $f$ with respect to $\sigma$. Note thereby that $k^*=k^{-1}$ for all $k \in K$. It is easy to see that $\End_\rho(V)$ is invariant under the map $f \mapsto f^*$. Now let $f$ be a $K$--equivariant endomorphism. If $l\in L$ then by the uniqueness of $L$ as the fixed point space of $K$ we obtain $f(l)\in L$. For $w \in W$ we compute
\[
 \sigma(f(w),l)=\sigma(w,f^*(l))=0
\]
for all $l \in L$ since $W$ is the orthogonal complement to $L$, see Proposition \ref{PRepresentationSO(2)}. Hence $f|W \colon W \to W$ is a $K$--equivariant map. Fix a vector $w_0 \in W$ then for every $w \in W$ there are $\lambda \geq 0$ and $k \in K$ such that $w=\lambda( k.w_0)$. This implies $f|W=\lambda \cdot k|W$
\end{proof}

We apply Proposition \ref{PRepresentationSO(2)} to the case where $\rho$ is the isotropy representation of $(M,G)$ in some point of $M$. This leads to a $1$--dimensional Distribution $\mathcal D$ on $M$ such that $\mathcal D_m$ is the $1$--dimensional fixed point subspace of $TM_m$ as in Proposition \ref{PRepresentationSO(2)}. Moreover this line bundle is trivial:

\begin{cor}
There is a $G$--invariant vector field $X$ on $M$ which has no zeros and for every point $m \in M$ the vector $X_m$ spans  the space $\mathcal D_m$. 
\end{cor}

\begin{proof}
Fix a point $m_0 \in M$ and a $0\neq \xi \in \mathcal D_{m_0}$. For $m \in M$ there is a $g \in G$ such that $m=g.m_0$. Define now $X_m:=g.\xi$ where here we use the induced action of $G$ on $TM$. This is well--defined since any other $g' \in G$ such that $m=g'.m_0$ differs from $g$ by an element of the isotropy $K$ in $m_0$. But $\xi$ is fixed by elements of $K$, hence $X_m$ is well--defined. And since the isotropy groups of the points $m$ and $m_0$ are conjugated by $g$ we obtain that $X_m$ lies in $\mathcal D_m$.
\end{proof}

\begin{rem}~
\begin{enumerate}
 \item[(a)] The vector field $X$ is complete because of his $G$--invariance and moreover its flow commutes with all elements of $G$, compare Remark \ref{RCommutingActions}. So let $\Phi \colon \RR \times M \to M$ be its flow. This can be viewed as an action of $\RR$ on $M$. We will show in Proposition \ref{PDivergenceOfX} that the integral curves of $X$ are geodesics for every $G$--invariant metric. 
 \item[(b)] If $\mu$ is a $G$--invariant metric, the length of $X$, $\sqrt{\mu(X,X)}$, as well as the divergence of $X$ with respect to $\mu$ are constant functions on $M$. Surely $\mu(X,X)$ is non--zero and $\div X$ is zero if $(M,G)$ admits compact quotients (cf. \cite{Thurston1997}). Moreover as the next proposition will show, $\div X$ does not depend on the $G$--invariant metric and indicates when $X$ is a Killing field.
\end{enumerate}
\end{rem}

\begin{prop}\label{PDivergenceOfX}
The divergence of $X$ with respect to a $G$--invariant metric does not depend on the metric and $\div X=0$ if and only if $X$ is a Killing field. Furthermore the integral curves of $X$ are geodesics for $G$--invariant metrics.
\end{prop}

\begin{proof}
Let $\mu$ be a $G$--invariant metric. The metric $\mu$ induces a global $G$--invariant Riemannian volume form $v_\mu$ since $M$ is orientable, which follows from the fact that $(M,G)$ has connected isotropy groups. Thus if $\mu'$ is another $G$--invariant metric and $v_{\mu'}$ its corresponding $G$--invariant volume form, then there exists a non--zero scalar $\lambda \in \RR$ such that $ v_{\mu'}=\lambda v_\mu$. Hence
\[
 (\div_{\mu' }X)v_{\mu'}=\mathcal L_X v_{\mu'}=\lambda \mathcal L_Xv_\mu=(\div_\mu X)v_{\mu'}
\]
which shows that $\div_\mu X$ does not depend on $\mu$.
 
Fix a point $m_0 \in M$ and let $K$ be the isotropy group in $m_0$. The bilinear form $\mathcal L_X \mu$ on $M$ is $G$--invariant and hence it is determined by the symmetric bilinear form $\beta:=(\mathcal L_X\mu)_{m_0}$. Let $V:=TM_{m_0}=L \oplus W$ be the $K$--invariant decomposition (see Proposition \ref{PRepresentationSO(2)}) and denote by $f \colon V \to V$ the $K$--invariant endomorphism defined as $\mu_{m_0}(f(v),w)=\beta(v,w)$ for all $v,w \in V$, which is self-adjoint with respect to $\mu_{m_0}$. Note that since $\div X$ is constant we have $2\div X=\tr f$. With Proposition \ref{PKInvariantEndos}  it follows that $f$ is represented as
\[
 \left(
 \begin{matrix}
 \lambda_1 & 0 \\
 0 & \lambda_2\id_W \\
 \end{matrix}
 \right) 
\]
where $\lambda_1,\lambda_2 \in \RR$. For all $l \in L$ we have $\beta(l,l)=0$ thus $\lambda_1 =0$. Hence the vector field $X$ is a Killing field if and only if $\lambda_2=0$. But this is equivalent to $\div X=0$ since $\div X=\frac{1}{2}\tr f=\lambda_2$.

Suppose $\mu$ is a $G$--invariant metric on $(M,G)$ and $\nabla$ its Levi--Civita connection. We have $k.(\nabla_{X_{m_0}} X)=\nabla_{X_{m_0}} X$ for all $k \in K$, so there is a $\lambda \in \RR$ such that $\nabla_{X_{m_0}} X=\lambda X_{m_0}$. But since $X$ has constant length with respect to $\mu$, the vector $\nabla_{X_{m_0}} X$ lies orthogonal to $X_{m_0}$, hence $\lambda=0$. This shows that the integral curves of $X$ are indeed geodesics since $\nabla_X X$ is $G$--invariant.
\end{proof}

Everything is set up to prove that the quotient of $M$ by the flow of $X$ is indeed a smooth manifold.

\begin{prop}
The action of $\RR$ through the flow of $X$ is either free or descends to a free circle action of $S^1$ on $M$.
\end{prop}

\begin{proof}
Suppose the action of $\RR$ is not free. Let $Z_m$ be the isotropy group of $\RR$ in the point $m \in M$ and since $Z_m$ has to be a closed $0$--dimensional subgroup of $\RR$ we obtain that $Z_m$ is isomorphic to $\ZZ$. The group $Z_m$ fixes now any other point on $M$ since the flow $\Phi$ and the action of $G$ commute. Hence we may pass to an induced free action $\Psi \colon \RR/\ZZ \times M \to M$. 
\end{proof}

In the case of a circle action we get a smooth surface $N$ as the quotient $M/S^1$ (i.e. the leaf space of the foliation) and moreover the quotient map $\pi \colon M \to N$ is a circle principal bundle over $N$ since $S^1$ is compact. We would like to obtain a similar result in the case of a free $\RR$ action on $M$. The key property is thereby that $\RR$ has to act properly on $M$, i.e. the map $f \colon \RR \times M \to M \times M$, $(t,m)\mapsto (t.m,m)$ has to be a proper map between topological spaces.

To prove this it is sufficient to show that if $(m_i)_{i \in \NN}$ is a converging sequence in $M$ and $(t_i.m_i)_{i\in \NN}$ converges as well then $(t_i)$ has a converging subsequence. From now on we assume a free $\RR$ action on $M$. We define $\Fix(K)=\{m \in M : k.m=m \text{ for all } k \in K\}$ the fixed--point set of $K$. For every $G$--invariant metric $\Fix(K)$ is a totally geodesic, embedded and closed submanifold of $M$, see \cite[p.59]{Kobayashi1995}.

\begin{prop}\label{POrbitsAreFixedPointSets}
The orbit $\RR.m_0$ is the connected component of $\Fix(K)$ containing $m_0$.
\end{prop}

\begin{proof}
$F$ shall denote the connected component of $\Fix(K)$ containing $m_0$. Then since the actions of $G$ and $\RR$ commute (see Remark \ref{RCommutingActions}) we have $\RR.m_0 \subset F$. Choose now a $G$--invariant metric and let $m \in F$. So there is a geodesic $\gamma$ lying in $F$,  starting at $m_0$ and ending in $m$. Moreover there is a $\lambda  \in \RR$ such that $\dot\gamma(0)=\lambda X_{m_0}$ and therefore $\gamma(t)=\exp_{m_0}(t\lambda X_{m_0})=(t\lambda).m_0$ since $\beta :\colon t\mapsto t.m_0$ is the geodesic with $\dot\beta(0)=X_{m_0}$. This shows $F\subset \RR.m_0$.
\end{proof}

\begin{cor}\label{COrbitsAreEmbedded}
The orbits $\RR.m$ for $m \in M$ are embedded, closed submanifolds of $M$.
\end{cor}

\begin{cor}\label{CProperAction}
The action of $\RR$ through the vector field $X$ is proper.
\end{cor}

\begin{proof}
Let $(m_i)$ be a converging sequence in $M$ and let $(t_i)$ be a sequence in $\RR$ such that $(t_i.m_i)$ converges as well. We will show that $(t_i)$ has a converging subsequence which implies that the action is proper. Choose a sequence $(g_i)$ in $G$ such that $m_i=g_i.m_0$ for all $i \in \NN_0$. Moreover the principal bundle $p \colon G \to M$, $p(g)=g.m_0$ has compact fibers and since $(m_i)$ converges, we may pass to a subsequence such that $(g_i)$ converges in $G$ which implies that $g^{-1}_i.(t_i.m_i)$ converges. If we use that the actions of $G$ and $\RR$ commute we see that $g^{-1}_i.(t_i.m_i)=t_i.m_0$ converges and with Corollary \ref{COrbitsAreEmbedded} we obtain that $(t_i)$ converges to some point in $\RR$.
\end{proof}

Now by the Quotient Manifold Theorem (see \cite[p. 218]{Lee2003}) the space $M/\RR$ has a unique manifold structure such that the canonical projection $\pi \colon M \to M/\RR$ is an $\RR$-principal bundle. This motivates the following 

\begin{dfn}
Let $H \in \{\SOg(2),\RR\}$. We set $N:=M/H$ through the action considered above. Moreover let $\pi \colon M \to N$ denote the $H$--principal bundle over the smooth surface $N$.
\end{dfn}

Applying the long exact homotopy sequence to the fibration $H \to M \to N$ we conclude that $N$ has to be simply connected since $H$ is connected and $M$ simply connected.

\subsection{The $G$--invariant connection on $M \to N$}\label{STheGInvariantConnectionOnMtoN}

Now we are ready to study the canonical $G$-invariant connection on the principal bundle $H \to M \to N$.

Let $\mu$ be a $G$--invariant metric on $M$ and let $\omega$ be the dual form of $X$. Proposition \ref{PRepresentationSO(2)} tells us that the kernel of $\omega$ does not depend on the choice of the $G$--invariant metric $\mu$. Moreover, since $\mu$ and $X$ are $G$--invariant we obtain that $\omega$ is $G$-invariant as well and we may choose $\mu$ such that $\omega(X)=1$. 

We start with some technical propositions which will be needed to prove that $\omega$ is indeed a connection.

\begin{rem}\label{RCommutingActions}
Let $\Phi\colon \RR \times M \to M$ be the flow of $X$. For every $g \in G$ the curves $t\mapsto \Phi^t(g.m)$ and $t\mapsto g.\Phi^t(m)$ are both integral curves of $X$ and they coincide in $t=0$. Hence $\Phi^t(g.m)=g.\Phi^t(m)$ for all $t \in \RR$ which shows that the action of $G$ and the flow of $X$ commute which in turn implies that the actions of $H$ and $G$ commute.
Now if $m_0 \in M$, $F$ its orbit under $H$ and $K$ the isotropy group in $m_0$ then $K$ is the isotropy group for all points $m \in F$ since $H$ and $K$ commute.
\end{rem}

\begin{prop}\label{PGCommutesWithIsotropyGroup}
Let $m_0 \in M$, $F$ its orbit under $H$ and $K$ the isotropy group in $m_0$. The set 
\[
 K_F=\{g \in G : g.m_0 \in F\}
\]
is a closed, connected Lie subgroup of $G$ and there is a natural epimorphism $\Psi \colon K_F \to H$ with $\ker\Psi=K$. 
\end{prop}

\begin{proof}
Using that $G$ and $H$ commute we obtain that $K_F$ is indeed a subgroup of $G$. Moreover $K_F$ is closed since the orbit $F$ is closed (see Corollary \ref{COrbitsAreEmbedded}).  We define $\Psi$ as follows: for $g \in K_F$ there is a unique $h \in H$ such that $g.m_0=h.m_0$, so define $\Psi(g):=h$. Clearly $\Psi$ is continuous and it is a homomorphism of groups.  Namely, if $g_1, g_2 \in G$ and $h_i = \Psi(g_i)$ for $i=1,2$, then $(g_1g_2).m_0=(h_1h_2).m_0=(\Psi(g_1)\Psi(g_2)).m_0$. Hence $\Psi$ is a Lie group homomorphism. $\Psi$ is onto since $G$ acts transitively and the kernel of $\Psi$ is the isotropy group of the $G$--action in $m_0$, namely $K$.

It remains to show that $K_F$ is connected. Therefore note that $K_F/K=H$ and since $K$ as well as $H$ are connected groups it follows also that $K_F$ has to be  connected.
%
\end{proof}

\begin{cor}\label{CElementsOfKFCommuteWithK}
All elements of $K_F$ commute with $K$, with other words the extension 
\[
 1 \longrightarrow \SOg(2) \longrightarrow K_F \longrightarrow H \longrightarrow 1
\]
is a central extension.
\end{cor}
\begin{proof}
Consider the (smooth) conjugation map $C \colon K_F \to \Aut(\SOg(2))$, $C_g(k)=gkg^{-1}$. $C$ well--defined since $K$ fixes all points on $F$. But $\Aut(\SOg(2))\cong\ZZ_2$ and since $K_F$ is connected we obtain $C_g=\id_{\SOg(2)}$ for all $g \in K_F$.
\end{proof}

\begin{cor}\label{CGinvariantConnection}
The $1$--form $\omega$ is a $G$--invariant connection form for $\pi \colon M \to N$.
\end{cor}

\begin{proof}
The Lie algebra $\mathfrak h$ of $H$ is spanned by the vector field $X$ and therefore we may identify $\mathfrak h$ with $\RR$ through $X$. In that way $\omega$ takes values in $\mathfrak h$ and the fundamental vector field of $r \in \RR$ is given by $rX$, hence $\omega(rX)=r$. We check now that $\omega$ is invariant under the action of $H$. For $h \in H$ and $\xi \in TM_{m_0}$ we obtain
\[
 h^*(\omega)_{m_0}(\xi)=\omega_{h.m_0}(h.\xi)=\mu_{h.m_0}(X_{h.m_0},h.\xi).
\]
Now choose $g \in G$ such that $g.h.m_0=m_0$ thus $g \in K_F$ where $F$ is the orbit of $m$ under $H$. Let $f \colon TM_{m_0} \to TM_{m_0}$ be the derivative of the diffeomorphism $m\mapsto g.h.m$ in $m_0$. We obtain
\[
 \mu_{h.m_0}(X_{h.m_0},h.\xi)=\mu_{m_0}(X_{m_0},f(\xi)).
\]
Using Corollary \ref{CElementsOfKFCommuteWithK} we see that $f$ commutes with elements of $K$, i.e. $f$ is a $K$--equivariant map. Using Proposition \ref{PKInvariantEndos} we obtain $\mu_{m_0}(X_{m_0},f(\eta))=0$ for $\eta \in \ker\omega_{m_0}$ and we conclude $\omega$ is indeed a connection form on $\pi \colon M \to N$.
\end{proof}

The fiber group $H$ of $\pi \colon M \to N$ is abelian and therefore the exterior derivative $\de\omega$ represents the curvature of $\omega$. Note since $\omega$ is $G$--invariant the curvature $\de\omega$ is $G$--invariant as well. Hence if $\de\omega$ is zero in a single point it is zero everywhere, i.e. $\omega$ is flat. We would like to study the flat case first. 

\subsection{Classification of geometries with flat connections}\label{SFlatConnection}

In this part we classify the axially symmetric geometries where the canonical  $G$--invariant connection is flat. The crucial fact is to deduce an extension of $G$ from the flatness of $\omega$.

Now since $\omega$ is flat and $M$ as well as $N$ are simply connected we obtain the 

\begin{cor}\label{CTrivialBundleIfFlat}
If $\omega$ is flat the $H$--principal bundle $\pi \colon M \to N$ is isomorphic to the trivial $\RR$--principal bundle $\pr_1 \colon N\times \RR \to N$ endowed with the trivial connection $\de t$ where $t:=\pr_2 \colon N \times \RR \to \RR$ is the projection onto the second factor. Moreover we have
\begin{enumerate}
 \item[(a)] The horizontal distribution is given by $\mathcal H=\ker\de t$ and a $G$--invariant vector field is obtained by taking the vector field to the flow $\Phi \colon \RR \times (N\times\RR) \to N \times \RR$, $\Phi^t(n,s)=(n,s+t)$ which shall be denoted by $\partial_t$.
 \item[(b)] For every $G$--invariant metric $\ker\de t$ is orthogonal to $\partial_t$ and $\de t$ has to be $G$--invariant.
\end{enumerate}

\end{cor}

\begin{prop}\label{PExtensionInFlatCase} 
There is a Lie group epimorphism $\Pi \colon G \to \RR$ such that $G':=\ker\Pi$ acts on $N\times 0$ effectively, transitively and with isotropy group $K$.
\end{prop}

\begin{proof}
We define $\Pi \colon G \to \RR$ similar to the homomorphism $\Psi$ in Proposition \ref{PGCommutesWithIsotropyGroup}. Let $t \colon N \times \RR \to \RR$ be the projection onto the second factor. The map $g_t \colon N \times \RR \to \RR$, $g_t:=t\circ g$ for $g \in G$ does not depend on the $N$--factor since $g$ respects the trivial connection. Hence the map $\Pi \colon G \to \RR$, $\Pi(g):=g_t(n,0)$ is well--defined where $n \in N$ is arbitrary. We claim that $\Pi$ is a Lie group homomorphism. If $g_1,g_2 \in G$ we set $s_i:=(g_i)_t(n,0)$ and $n_2:=\pr_1\circ g_2(n,0)$ (note that $s_i$ does not depend on the choice of $n \in N$). We obtain
\[
(g_1g_2)(n,0)=g_1(n_2,s_2)=\Phi^{s_2}(g_1(n_2,0))
\]
since $\Phi$ and elements of $G$ commute. Hence 
\[
 \Pi(g_1g_2)=s_1+s_2=\Pi(g_1)+\Pi(g_2).
\]
Moreover $\Pi$ is continuous since it is the evaluation map in $(n,0)$ and therefore smooth. For $s \in \RR$ choose an $n \in N$ and a $g \in G$ such that $g(n,0)=(n,s)$. We have $\Pi(g)=s$ and this shows that $\Pi$ is indeed onto.
It remains to study the kernel.

Clearly $G':=\ker\Pi$ acts on $N\times 0$ transitively. Suppose $g \in G'$ and $g(n,0)=(n,0)$ for all $n \in N$. It follows for all $(n,t) \in N \times\RR$
\[
 g(n,t)=g\circ \Phi^t(n,0)=\Phi^t\circ g(n,0)=(n,t)
\]
which implies $g=\id$. Let $K$ be the isotropy group in $(n,0)$ of $G$ acting on $N \times \RR$. Then $K$ is a subgroup of $G'$ and hence $K$ is the stabilizer of $G'$ acting on $N\times 0$ in that point.
\end{proof}

\begin{rem}
From Proposition \ref{PExtensionInFlatCase} we know that $N$ is the homogeneous space $G'/K$, where $G'$ is defined as above and $K$ is an isotropy of $G$ in a point on $N \times 0$. And since $\Pi$ is onto, $G'$ is of dimension $3$ and connected which follows from the long exact homotopy sequence.
The spaces $G'/K$ admit a homogeneous metric which has constant Gaussian curvature, hence they are two--dimensional space forms. Thus if we identify $N$ with $N \times 0$ the pair $(N,G')$ is equivariant diffeomorphic to exactly one of the following three pairs
\[
 (\RR^2,\Eg_0(2)),\; (S^2,\SOg(3)),\; (D^2, \SOg^+(2,1)).
\]
\end{rem}
\bigskip
One important step to analyze the short exact sequence 
\[
1\longrightarrow G'\longrightarrow G\longrightarrow \RR\longrightarrow 1
\]
from Proposition \ref{PExtensionInFlatCase} is to determine the $G$--invariant metrics on $M$. For that purpose we would like to introduce some notations. Suppose $(N \times \RR,G)$ is a geometry like in Corollary \ref{CTrivialBundleIfFlat} and let $T(N\times \RR)=\mathcal H \oplus \mathcal V$ where $\mathcal H=\ker(\pr_2)_*$ is the trivial connection and $\mathcal V$ is the line bundle spanned by $\partial_t$. Let $\e_\kappa$ be the pullback on $N \times \RR$ of the function $\RR \to \RR$, $t\mapsto e^{\kappa t}$, where $\kappa=\div\partial_t$ with respect to any $G$-invariant metric, see Proposition \ref{PDivergenceOfX}.

\begin{prop}\label{PGInvariantMetricsFlatCase}
Let $(N\times \RR, G)$ be a flat geometry like in Corollary \ref{CTrivialBundleIfFlat} and let $T(N\times \RR)=\mathcal H \oplus \mathcal V$ as described above. If $\mu$ is a $G$--invariant metric on $N \times \RR$ then
\[
\mu=\e_\kappa\nu + \lambda\,\de t^2
\]
where $\nu$ is the pullback by $\pr_1$ of a metric of constant Gaussian curvature on $N$ and $\lambda=\mu(\partial_t,\partial_t)$ and $\kappa = \div\partial_t$.

\end{prop}

\begin{proof}
By Corollary \ref{CTrivialBundleIfFlat} $\mathcal H$ and $\mathcal V$ are orthogonal for any $G$--invariant metric, hence $(\Phi^t)^*(\mu)$ must have the form $\nu^*+\lambda \de t^2$
where $\nu^*$ is a symmetric bilinear form such that $\nu^*(\partial_t,\xi)=0$ for all $\xi \in T(N\times \RR)$ and $\lambda=\mu(\partial_t,\partial_t)$ (note that $\de t(\partial_t)=1$). Fix a point $(n,s) \in N\times \RR$ and let $t\mapsto g_t$ be a smooth curve in $G$ such that $g_t\circ \Phi^t(n,s)=(n,s)$ for all $t$ and $g_0 =\id$. Denote by $f_t$ the automorphism of $T(N\times\RR)_{(n,s)}$ which is given as the derivative of $g_t\circ\Phi^t$ in $(n,s)$. We obtain
\begin{equation}\label{G1}
  (\Phi^t)^*(\mu)_{(n,s)}=(f_t)^*(\mu_{(n,s)})
\end{equation}
for all $t$. Using Corollary \ref{CElementsOfKFCommuteWithK} and Remark \ref{RCommutingActions} we see that $f_t$ commutes with all elements of $K$. And moreover, using Proposition \ref{PKInvariantEndos} we have $f_t|\mathcal H_{(n,s)}=r_tk_t|\mathcal H_{(n,s)}$ with $r_t>0$ and $k_t \in K$. Taking the derivative of equation (\ref{G1}) we obtain for $\zeta_1,\zeta_2 \in \mathcal H_{(n,s)}$
\[
 (\Phi^t)^*(\mathcal L_x\mu)(\zeta_1,\zeta_2)=2r_t\dot r_t\mu(\zeta_1,\zeta_2)
\]
where the dot denotes the $t$--derivative of $t\mapsto r_t$. From the proof of Proposition \ref{PDivergenceOfX} we know that $\mathcal L_X\mu(\zeta_1,\zeta_2)=\kappa\mu(\zeta_1,\zeta_2)$ and so we end up to solve the equation
\[
 \kappa r_t = \tfrac{1}{2}\dot r_t
\]
such that $r_0=1$, hence $r_t=e^{\frac{1}{2}\kappa t}$. 

Moreover we have $(\De\Phi^t)_{(n,s)} \colon TN_n\oplus \RR(\partial_t)_{(n,s)}\to  TN_n\oplus \RR(\partial_t)_{(n,s+t)}$
\[
 (\De\Phi^t)_{(n,s)}(\zeta,(\partial_t)_{(n,s)})=(\zeta,(\partial_t)_{(n,s+t)})
\]
and therefore for $\zeta_1,\zeta_2 \in \mathcal H_{(n,s)}$
\[
 \mu_{(n,s)}(\zeta_1,\zeta_2)=(\Phi^s)^*(\mu)_{(n,0)}(\zeta_1,\zeta_2)=e^{\kappa s}\mu_{(n,0)}(\zeta_1,\zeta_2).
\]
But $\mu$ restricted to $N \times 0$ is invariant under $G'$, cf. Proposition \ref{PExtensionInFlatCase}. Hence it is a metric of constant curvature on $N$. Let $\nu$ be its pullback under $\pr_1$. We obtain
\[
 \mu_{(n,s)}(\zeta_1,\zeta_2)=e^{\kappa s}\mu_{(n,0)}(\zeta_1,\zeta_2)=e^{\kappa s}\nu_{(n,s)}(\zeta_1,\zeta_2).
\]
\end{proof}

Before we state the classification for geometries with flat canonical connections we would like to introduce a geometry, which is less known and which is an axially symmetric one with flat connection.

\begin{rem}\label{RCompareToThurston}
If we assume that $G$ admits cocompact discrete subgroups then with Proposition \ref{PDivergenceOfX} $t \mapsto \Phi^t$ is a one-parameter group of isometries for any $G$-invariant metric on $M$. Moreover, $\Phi^t \in G$ for all $t \in \RR$ since $\Phi^t$ maps $\partial_t$ to $\partial_t$. By Corollary \ref{CTrivialBundleIfFlat} the flow is just a translation in the fiber, hence $t\mapsto \Phi^t$ is a section for $\Pi \colon G \to \RR$. It follows that $G$ is isomorphic to $G'\times \RR$ since $\Phi^t$ lies in the center of $G$. The resulting geometries are the one listed in \cite[p. 183]{Thurston1982}.
\end{rem}

\begin{exmpl}\label{EFlatGeometry}
Let $\k \in \RR$ and set $G_\k:=\Eg_0(2)\rtimes_{\rho_\k} \RR$, where $\rho_\kappa \colon \RR \to \Aut_0(\Eg_0(2))$ is given as $(\rho_\k)_s(a,A):=(e^{-\frac{1}{2}\kappa s}a, A)$.  For the sake of convenience we omit the subscript $\kappa$ for $\rho_\k$ and $G_\k$ but nevertheless we keep in mind that those objects depend on that number. Consider the manifold $M=\RR^2\times \RR$ where we denote by $(x,t)$ a point in $M$ such that $x \in \RR^2$ and $t \in \RR$. The group $G$ acts on $M$ by
\[
 (a,A,s).(x,t):=\left( e^{-\frac{1}{2}\kappa s} Ax +a,t+s\right).
\]
It is easy to see that $(M,G)$ is an axially symmetric geometry and the $G$--invariant vector field $\partial_t$ is obtained as the vector field to the flow $\Phi^t(x,s)=(x,s+t)$. The metric
\[
 \mu_{(x,t)}=
 \left(
 \begin{matrix}
 e^{\k t}\cdot E_2 & 0 \\
 0 & 1 \\
 \end{matrix}
 \right)
\]
is invariant under the action of $G$. The Riemannian volume form $v_\mu$ with respect to the standard orientation $\de x\wedge \de t$ is $e^{\k t} \de x \wedge \de t$ and it follows that 
\[
 \mathcal L_{\partial_t}v_\mu=\kappa v_\mu,
\]
hence $\div\partial_t =\k$. Moreover the connection form is $\de t$ which is a flat connection on the trivial $\RR$--principal bundle $\RR^2 \times \RR \to \RR^2$.

If $\kappa \neq 0$ then $(M,G_\k)$ is equivariant diffeomorphic to $(M,G_1)$. The map $f \colon M \to M$, $f(x,t)=(x,\k t)$ is an $F$--equivariant diffeomorphism, where $F \colon G_\k \to G_1$, $F(a,A,s)=(a,A,\k s)$. The geometries $(M,G_1)$ and $(M,G_0)$ cannot be isomorphic since there is no Lie group isomorphism between $G_1$ and $G_0$ (note that $G_1$ has trivial center and $G_0$ not). 

Clearly $G_0$ is a subgroup of $\Eg(3)$. The group $G_1$ is a subgroup of the isometry group of hyperbolic $3$--space, $\SOg^+(3,1)$. To see this consider the upper half--space $H=\RR^2 \times \RR_{>0}$ and the diffeomorphism $F \colon H \to M$, $(x,s)\mapsto (x,-\ln s)$. The group $G_1$ acts on $H$ via $F$ and the metric $F^*(\mu)$ is invariant under this action. But $F^*(\mu)$ is the standard hyperbolic metric on $H$ of constant curvature equal to $-1$, hence $G_1 \subset \SOg^+(3,1)$.
\end{exmpl}

\begin{thm}\label{TClassificationFlatGeometries}
Let $(M,G)$ be an axially symmetric geometry and $\pi \colon M \to N$ the induced principal bundle with connection form $\omega$. Suppose $\omega$ is flat. Then $(M,G)$ is equivariant diffeomorphic to exactly one of the following geometries
\begin{enumerate}
 \item[(a)] $(S^2 \times \RR, \SOg(3) \times \RR)$
 \item[(b)] $(D^2\times \RR, \SOg^+(2,1) \times \RR)$
 \item[(c)] $(\RR^2 \times \RR, \Eg_0(2) \times \RR)$
 \item[(d)] $(\RR^2 \times \RR, \Eg_0(2) \rtimes_{\rho_1} \RR)$ \hfill [cf. Example \ref{EFlatGeometry}]
\end{enumerate}

\end{thm}

\begin{proof}
From Corollary \ref{CTrivialBundleIfFlat} we know that $(M,G)$ is isomorphic to $(N \times \RR, G)$ and the principal bundle $\pi \colon M\to N$ is the trivial $\RR$--bundle $N \times \RR \to N$. From Proposition \ref{PExtensionInFlatCase} we have an exact sequence 
\[
 1 \lr G' \lr G \lr \RR \lr 1.
\]
such that $G'/K=N$. This sequence always splits as $\RR$ is simply connected. This implies that $G$ is isomorphic to the semidirect product $G' \rtimes_\rho \RR$, for a Lie group homomorphism $\rho \colon \RR \to \Aut_0(G')$. The isotropy group $K$ of $G$ acting on $N \times \RR$ lies in $G'\rtimes \RR$ as $K \times 0$. For $G'$ we have the possibilities $\SOg(3)$, $\SOg^+(2,1)$ and $\Eg_0(2)$.

Suppose first $G'=\SOg(3)$ thus $N =S^2$. The group $\SOg(3)$ is complete (i.e. $\SOg(3)$ is centerless and the autormorphism group is equal to the inner automorphism group)  and therefore $\rho$ is the trivial action on $G'$ which means that $G$ is $\SOg(3)\times \RR$. Since the isotropy group is $K \times 0$ we have that $(M,G)$ is isomorphic to $(S^2\times \RR, \SOg(3) \times \RR)$ with the standard action. If $G' =\SOg^+(2,1)$ then the connected component of its automorphism group is $\SOg^+(2,1)$ itself. Same arguments as for the case $G'=\SOg(3)$ apply since $\rho$ is continuous and thus the geometry is isomorphic to $(D^2 \times \RR, \SOg^+(2,1) \times \RR)$ again endowed with the standard action.

The remaining case is $G'=\Eg_0(2)$ and $N=\RR^2$. If $\s \colon \RR \to G$ is the splitting map above then $\rho_t$ is the conjungation with $\sigma(t)$ in $G'$, $\rho_t=c_{\sigma(t)}$ for all $t \in \RR$. Hence the group structure of $G=\Eg_0(2)\rtimes_\rho \RR$ depends on the splitting map. We choose a $G$--invariant metric $\mu$ such that $\partial_t$ has length equal to $1$. With Proposition \ref{PGInvariantMetricsFlatCase} $\mu$ has the form
\[
 \mu=\e_\k\delta+\de t^2
\]
where $\delta$ is the pullback of a flat metric on $\RR^2$ to $\RR^2\times \RR$ and $\kappa = \div\,\partial_t$. For $t \in \RR$ consider the map 
$\sigma_t \colon \RR^2 \times \RR \to \RR^2 \times \RR$,
\[
 \s_t(x,s):=\left(e^{-\frac{\kappa }{2}t}x,s+t\right)
\]
which is an isometry for $\mu$ and furthermore $\sigma_t$ is contained in $G$ for all $t\in \RR$ (note that all isometries which maps $\partial_t$ to $\partial_t$ are elements of $G$). The map $\s \colon \RR \to G$, $t\mapsto \s_t$ is a Lie group homomorphism which is obviously a splitting map for the exact sequence above, since the last entry is a translation in the fiber. It follows that $\rho \colon \RR \to \Aut(\Eg_0(2))$ is given by the conjugation in the diffeomorphism group with $\s$
\[
 \rho_t(a,A)=\left( e^{-\frac{\kappa}{2}t}a, A\right)
\]
and therefore the action is given by
\[
 (a,A,t).(x,s)=\left(e^{-\frac{1}{2}\k t}Ax+a,t+s \right).
\]
This geometry was described in Example \ref{EFlatGeometry}.
\end{proof}

\subsection{Classification of geometries with non--flat connections}\label{SnonFlatConnection}

Henceforth we suppose $\de\omega\neq 0$. Using the $G$--invariance of $\de\omega$ we see that it is fully determined by $\delta:=\de\omega_{m_0}$ for $m_0 \in M$ and by assumption we have $\delta\neq 0$. From Proposition \ref{PKInvariantEndos} we deduce the

\begin{cor}\label{CCurvatureNonFlatCase}
With respect to the decomposition $TM_{m_0}=L\oplus W$ of Proposition \ref{PRepresentationSO(2)} $\delta$ has the form
\[
 \left(
 \begin{matrix}
 0 & 0 \\
 0 & v_2 \\
 \end{matrix}
 \right)
\]
where $v_2$ is a volume form on $W$.
\end{cor}

\begin{cor}\label{CXKillingField}
If $\de\omega\neq 0$ then $X$ is a Killing field for any $G$--invariant metric.
\end{cor}

\begin{proof}
With Proposition \ref{PDivergenceOfX} it is sufficient to show that $\kappa=\div X$ vanishes. Like in the proof of Proposition \ref{PGInvariantMetricsFlatCase} we obtain 
\[
 (\Phi^t)^*(\de\omega)=e^{\kappa t}\de\omega
\]
using the $G$--invariance of $\de\omega$. Hence $\mathcal L_X\de\omega=\kappa\de\omega$ but on the other hand we have $\mathcal L_X\de\omega=\de(i_X(\de\omega))$. Since $i_X(\de\omega)$ is $G$--invariant it is sufficient to evaluate this $1$--form in $m_0$. However with Corollary \ref{CCurvatureNonFlatCase} we see $i_{X_{m_0}}(\delta)=0$ which implies $\mathcal L_X\de\omega=0$. Using again Corollary \ref{CCurvatureNonFlatCase} $\de\omega$ restricted to the horizontal distribution is non--degenerated which forces $\kappa=0$. 
\end{proof}

\begin{prop}\label{PHisSubGroupOfG}
The fiber group $H \in \{\RR,\SOg(2)\}$ of the principal bundle $\pi \colon M \to N$ is a subgroup of the center of $G$ and $G':=G/H$ acts on $N$ effectively, transitively with isotropy group isomorphic to $\SOg(2)$ such that $\pi$ is an equivariant map.
\end{prop}

\begin{proof}
Let $\mu$ be a $G$--invariant metric. The flow of $t\mapsto \Phi^t$ is a one--parameter subgroup of the isometry group of $(M,\mu)$ which we deduce from Corollary \ref{CXKillingField}. But since $\Phi^t$ maps $X$ into $X$ it is easy to see that $\Phi^t$ is indeed an element of $G$. We saw in Remark \ref{RCommutingActions} that the actions of $G$ and $H$ commute which means that $H$ is a subgroup of the center of $G$.

The group $G'=G/H$ acts on $N=M/H$ by the induced action of $G$ on $M$: if $[g] \in G/H$ and $[m] \in M/H$ then $[g].[m]:=[g.m]$ which is well--defined. Clearly $G'$ acts transitively on $N$ since $G$ does on $M$. Suppose for $g' \in G'$ we have $g'.n=n$ for all $n \in N$. If $g\in G$ represents $g'$ then there is a $h \in H$ such that $g.m_0=h.m_0$, so there is a $k \in K_{m_0}$ with $g=hk$. But this would imply that $[k]$ acts as the identity on $N$. We claim that this means that $k$ has to be the identity element, which proves that $G'$ acts effectively. Let $TM_{m_0}=L \oplus W$ be the decomposition of Proposition \ref{PRepresentationSO(2)}. By construction the differential of $\pi \colon M \to N$ in $m_0$ namely $\De\pi_{m_0}|W \colon W \to TN_{n_0}$ ($n_0=\pi(m_0)$) is an isomorphism. The assumption $\pi \circ k =\pi$ implies  $k.\xi=\xi$ for all $\xi \in W$. Thus $k$ has to be the identity on $M$ and therefore $g=h$.

The isotropy group $K'$ in $n_0=\pi(m_0)$ of $G'$ acting on $N$ is the image of $K$ under the quotient map $p \colon G \to G'$. Note that $K'$ has to be connected and $1$--dimensional for the same reasons $K$ is connected. Clearly we have $p(K)\subset K'$, so we have to show that $p|K$ is injective. For $k \in K$ suppose $p(k)=e$. Then, as above, we have $k \in H$ which implies that $k$ is the neutral element, since $k$ has a fixed point but $H$ is acting freely on $M$.
\end{proof}

As in the case of a flat connection there are three possibilities for the pair $(N,G')$. Moreover Proposition \ref{PHisSubGroupOfG} implies a central extension
\[
1 \lr H \lr G \lr G' \lr 1.
\]
We handle each case of $(N,G')$ separately, since different techniques are required to determine the geometries $(M,G)$. We start with the spherical case, i.e. $(N,G')=(S^2,\SOg(3))$.

\begin{prop}\label{PGeometryS3}
Let $(N,G')=(S^2,\SOg(3))$. Then $(M,G)$ is isomorphic to $(S^3, \Ug(2))$.
\end{prop}
\begin{proof}
First we determine the fiber group. Suppose that $H=\RR$. Then there would be a global section $s \colon S^2 \to M$ and $s^*(\de\omega)$ would be a volume form on $S^2$, see Corollary \ref{CCurvatureNonFlatCase}. But this leads to a contradiction by Stokes' theorem since $\de(s^*(\omega))=s^*(\de\omega)$ is exact, hence $H=\SOg(2)$. The central extension 
\[
 1 \lr \SOg(2) \lr G \lr \SOg(3) \lr 1
\]
implies a central extension for their Lie algebras
\[
 0 \lr \RR \lr \fg \lr \fso(3) \lr 0.
\]
Using Whitehead's second lemma (see \cite[p.246]{Weibel1997}) the central extension splits and is isomorphic to the trivial one, since it is central. Hence the universal covering group of $G$ is $\RR \times \SUg(2)$ and $G\cong(\SUg(2)\times \RR)/\pi_1(G)$ where $\pi_1(G)$ is a discrete subgroup of center $\RR\times \ZZ_2$. Applying the long exact homotopy sequence to the principal bundles $M \to S^2$ and $G\to \SOg(3)$ we obtain that $\pi_1(G)$ is isomorphic to $\ZZ$. Let $(c,d)$ be a generator of $\pi_1(G)$ for $c \in \RR$ and $d \in\ZZ_2=\{\pm 1\}$. Then $c\neq 0$ since $\pi_1(G)$ is isomorphic to $\ZZ$. Moreover $d=-1$ since otherwise $G$ would be isomorphic to $\SOg(2) \times \SUg(2)$ and the quotient $G/K$ could not be simply connected since $K$ has to lie in $\SUg(2)$ (note that $K\cap H$ has to be trivial). The map $\RR \times \SUg(2) \to \Ug(2)$, $(t,S)\mapsto e^{\frac{i\pi t}{c}}S$ induces an isomorphism between $(\RR \times \SUg(2))/\pi_1(G)$ and $\Ug(2)$, hence $G=\Ug(2)$. 

Since $K$ is isomorphic to $\Ug(1)$ we may assume (after a conjugation in $\Ug(2)$) that $K$ consists of diagonal matrices. To see this, note that $K$ is abelian and hence all elements are simultaneously diagonalizable. Let $\Phi \colon \Ug(1) \to K \subset \Ug(2)$ be an isomorphism. Then there are homomorphisms $\Phi_i \colon \Ug(1) \to \Ug(1)$ such that
\[
 \Phi(z) = 
 \left(
 \begin{matrix}
  \Phi_1(z) & 0 \\
  0 & \Phi_2(z) 
 \end{matrix}
 \right).
\]
Moreover there are $n,m \in \ZZ$ such that $\Phi_1(z)=z^n$ and $\Phi_2(z)=z^m$. Since $\Phi$ is an isomorphism we obtain
$(n,m)=(1,1)$, $(n,m)=(1,0)$ or $(n,m)=(0,1)$. But $n=m=1$ would mean, that $K$ lies in the center which is a contradiction to $K \cap H$ is trivial. The quotient of the remaining possibilities is $S^3$ and $\Ug(2)$ acting on it by its standard action on $\CC^2$.

%
\end{proof}

\begin{rem}\label{RTildeSL}
For the next case it is useful to remind some facts about the Lie group $\widetilde\SLg$ which we define as the universal cover group of $\SOg^+(2,1)$. The group $\SOg^+(2,1)$ can be identified as a manifold with the circle bundle of the standard hyperbolic plane which is topologically given as $D^2 \times S^1$, so  the fundamental group of $\SOg^+(2,1)$ is isomorphic to $\ZZ$. But this implies that the center of $\widetilde\SLg$ is isomorphic to $\ZZ$ since $\SOg^+(2,1)$ is centerless. Moreover this implies that $\widetilde\SLg$ is topologically $\RR \times D^2$ which fibers over $D^2$ such that the projection $\RR \times D^2 \to D^2$ is equivariant with respect to $\widetilde\SLg$ and $\SOg^+(2,1)$.

The group $\RR \times \widetilde\SLg$ has a natural action on $\widetilde\SLg$: elements of $\RR$ act as translations in the fibers of $\widetilde\SLg \to D^2$ which cover the identity and $\widetilde\SLg$ acts by group multiplication. This action descends to an action of $\Gamma:=(\RR\times\widetilde\SLg)/\ZZ(1,1)$ on $\widetilde\SLg$ where $(1,1)$ is an element of the center $\RR\times \ZZ$ such that $1 \in \ZZ$ generates the fundamental group $\pi_1(\SOg^+(2,1))\subset\widetilde\SLg$. A more detailed discussion of $(\widetilde\SLg,\Gamma)$ can be found in \cite{Scott1983} which is known as \emph{ $\widetilde\SLg$-geometry}.

\end{rem}

\begin{prop}\label{PGeometrySL}
If $(N,G')=(D^2, \SOg^+(2,1))$ then $(M,G)$ is isomorphic to the $\widetilde\SLg$--geometry $(\widetilde\SLg,\Gamma)$.
\end{prop}

\begin{proof}
Since $D^2$ is contractible the principal bundle $M \to N$ is the trivial $\RR$--principal bundle $\RR \times D^2 \to D^2$. This implies that $\pi_2(M)$ is trivial and using the long homotopy sequence for the principal bundle $\SOg(2) \to G \to M$ we obtain that $\pi_1(G)$ is isomorphic to $\ZZ$. Similarly to Proposition \ref{PGeometryS3}, we obtain the central extension
\[
 0 \longrightarrow \RR \longrightarrow \fg \longrightarrow \fsl(2,\RR)\to 0
\]
which has to be isomorphic to the trivial one
\[
 0 \longrightarrow \RR \longrightarrow \RR\times \fsl(2,\RR) \longrightarrow \fsl(2,\RR)\to 0
\]
since $\fsl(2,\RR)$ is semisimple. One then easily checks that the following diagram commutes
\[
  \begin{CD}
     @. 1 @. 1 @. 1\\
     @.	@VVV @VVV @VVV \\
     1 @>>> \pi_1(\RR) @>>> \pi_1(G) @>(\pr_2)_*>> \pi_1(\SOg^+(2,1)) @>>> 1\\
      @.	@VVV 	@VVV	@VVV \\
     1@>>> \RR @>>>  \RR\times\widetilde\SLg @>\pr_2>>  \widetilde\SLg@>>> 1 \\
     @.		   @VVV 		    @V\pi VV		      @V\pi_+ VV\\
      1  @>>> \RR @>>>		G	    @>p>> \SOg^+(2,1) @>>> 1.\\
     @.	@VVV @VVV @VVV \\
    @. 1 @. 1 @. 1\\
  \end{CD}.
\]
The center of $\RR\times \widetilde\SLg$ is isomorphic to $\RR \times \ZZ$ where $\ZZ$ is the group $\pi_1(\SOg^+(2,1)) \subset \widetilde\SLg$. Moreover  the projection map $\pr_2 \colon \RR \times \widetilde\SLg \to \widetilde\SLg$ induces an isomorphism on the fundamental groups $(\pr_2)_* \colon \pi_1(G) \to \pi_1(\SOg^+(2,1))$ which is simply the restriction of $\pr_2$ to $\pi_1(G)$ seen as a subgroup of $\RR\times\widetilde\SLg$. Hence a generator of $\pi_1(G)$ must be of the form $(c,1) \in \RR \times \ZZ$ where $1$ is a generator of $\ZZ=\pi_1(\SOg^+(2,1))$. However we may exclude $c=0$ since otherwise $G$ would be isomorphic to $\RR \times \SOg^+(2,1)$ which would be a geometry with flat $G$--invariant connection. Therefore we assume that $(1,1) \in \RR\times \ZZ$ is a generator of $\pi_1(G)$, since if $c\neq 1$ we may consider the automorphism $\Phi \colon \RR\times \widetilde\SLg \to \RR\times \widetilde\SLg$, $\Phi(x,a)=(x/c,a)$ which induces an isomorphism between $(\RR\times \widetilde\SLg/(\ZZ(
c,1))$ and $(\RR\times \widetilde\SLg/(\ZZ(1,1)
)$.

Let $\pi \colon \RR \times \widetilde\SLg \to G$ be the quotient map. Then $\widetilde \SLg$ acts on $\RR\times D^2$ via $\pi$ as a subgroup of $\RR \times \widetilde\SLg$. We claim that $\widetilde\SLg$ acts transitively on $\RR \times D^2$: Let $p \colon G \to \SOg^+(2,1)$ be the quotient map of $G \to G/H=\SOg^+(2,1)$ (see Proposition \ref{PHisSubGroupOfG}) and let $\pi_+ \colon \widetilde\SLg \to\SOg^+(2,1)$ be the universal cover map. Then we have $p \circ\pi =\pi_+\circ \pr_2$ since the diagram above commutes. For $(t,x) \in \RR\times D^2$ the action of $a \in \widetilde\SLg$ is then defined by $\pi(0,a).(t,x)$. The second entry is given by $\pi_+(a).x$ due to the fact that $p \circ\pi =\pi_+\circ \pr_2$ and since the  projection $\RR\times D^2 \to D^2$ is equivariant (see Proposition \ref{PHisSubGroupOfG}). Hence there is a map $f \colon \RR\times D^2 \times \widetilde\SLg \to \RR$ such that 
\[
\pi(0,a).(t,x) = (f(t,x,a),\pi_+(a).x) 
\]
with the property
\[
 f(t,x,ab)=f(f(t,x,b),\pi_+(b).x,a)
\]
for $a,b \in \widetilde\SLg$.
The differential of $\pi(0,a)$ as a map from $M$ to itself maps $\partial_t$ into 
itself since $\pi(0,a)\in G$. And this implies that if $x$ and $a$ are fixed the map $f$ is a translation in the fiber, i.e. there is a smooth function $s \colon D^2\times \widetilde \SLg \to \RR$ such that $f(t,x,a)=t+s(x,a)$. Then from the property of $f$ above we obtain
\[
 s(x,ab)=s(\pi_+(b).x,a) + s(x,b) \eqno (\ast)
\]
for $a,b \in \widetilde\SLg$. Note moreover that for  $z \in \pi_1(\SOg^+(2,1))\cong \ZZ$ with $z\neq 0$ the element $[(0,z)] \in G=(\RR\times\widetilde\SLg/\pi_1(G))$ is equal to $[(-z,0)]$, hence $(0,z)$ acts on $\RR \times D^2$ as constant translation, i.e. $\pi(0,z).(t,x)=(t-z,x)$ for all $(t,x)$.

Fix a point $x_0 \in D^2$ and let $K$ be the isotropy group in that point of $\SOg^+(2,1)$. The subgroup $\widetilde K :=(\pi_+)^{-1}(K)$ is isomorphic to $\RR$. To see this note the following: Let $Z=\pi_1(\SOg^+(2,1))\subset \widetilde\SLg$ then $\widetilde K / Z = K$ and $\widetilde \SLg/\widetilde K = \SOg^+(2,1)/K =D^2$, hence $\tilde K$ is connected. Furthermore $Z$ is a disrete subgroup of $\widetilde K$ which excludes the case $\tilde K=\SOg(2)$.

Now consider the map $\Psi \colon \widetilde K \to \RR$, $a \mapsto \Psi(a)=:s(x_0,a)$. Using the fact, that $\pi_+(\tilde K)$ fixes the point $x_0$, it follows from $(\ast)$ that $\Psi$ is a homomorphism. Since $\widetilde K$ is isomorphic to $\RR$, $\Psi$ is either trivial or an isomorphism. Suppose $\Psi$ is trivial. Then every element in $0\times Z=0\times \pi_1(\SOg^+(2,1))\subset \tilde 0\times K\subset \RR\times\widetilde\SLg$ would act as the identity on $\RR \times D^2$ which then in turn has to lie in $\ker\pi$ but this is a contradiction since the group $\ker\pi = \pi_1(G)$ is generated by $(1,1)$.

Finally it is easy to see that $\widetilde\SLg$ acts transitively on $\RR\times D^2$ and since the isotropy group has to be a point, we may identify $\RR\times D^2$ with $\widetilde \SLg$.

If we sum up the results in this proposition we obtain the same action of $\RR\times\widetilde\SLg$ on $\widetilde\SLg$ as for the $\widetilde\SLg$--geometry which descends to an action of $\Gamma$ on $\widetilde\SLg$.
\end{proof}

\begin{prop}\label{PGeometryNil}
Suppose $(N,G')=(\RR^2,\Eg_0(2))$. Then $(M,G)$ is isomorphic to the $\Nil$--geometry.
\end{prop}
\begin{proof}
The principal bundle $M \to \RR^2$ is the trivial $\RR$--principal bundle $\RR^2\times \RR \to \RR^2$ since $\RR^2$ is contractible. Repeating the arguments of the preceding two propositions we obtain $\pi_1(G)=\ZZ$ and the central extension of Lie Algebras
\[
 0 \lr \RR \lr \fg \lr \fe(2)\lr 0.
\]
The second Lie algebra cohomology group $H^2(\fe(2);\RR)$ (where $\RR$ is the trivial $\fe(2)$--module) is isomorphic to $\RR$. Every element in  $H^2(\fe(2);\RR)$ has a unique representative $\omega_\lambda$ given by
\[
 \left(
 \begin{matrix}
  0 & \lambda & 0 \\
  -\lambda & 0 & 0 \\
  0 & 0 & 0 \\
 \end{matrix}
 \right) 
\]
for $\lambda \in \RR$, which shows also the isomorphism between the cohomology group and $\RR$.
If $\omega_\lambda \in H^2(\fe(2);\RR)$ the isomorphism class of central extensions are of the form
\[
 0 \lr \RR \lr \RR \times_{\omega_\lambda} \fe(2) \lr \fe(2)\lr 0.
\]
(See \cite{Weibel1997} for the definition of the Lie Algebra $\RR \times_{\omega_\lambda} \fe(2)$.)

We consider first the trivial central extension
\[
  0 \lr \RR \lr \RR \times \fe(2) \lr \fe(2)\lr 0.
\]
The universal cover group $\widetilde{\Eg_0(2)}$ of $\Eg_0(2)$ is given as the semidirect product $\RR^2\times_\rho \RR$ where $\rho \colon \RR \to \GLg(\RR^2)$ and $\rho_\theta$ is the rotation around the origin with rotation angle $\theta$. Then the center is isomorphic to $\ZZ$ and embedded as $z\mapsto (0,2\pi z)$. The Lie group $\RR \times \widetilde{\Eg_0(2)}$ has Lie algebra $\RR \times \fe(2)$ and its center is isomorphic to $\RR \times \ZZ$. Like in Proposition \ref{PGeometrySL} a generator of $\pi_1(G)$ has to have the form $\gamma_c:=(c,0,2\pi)$ for $c \in \RR$. The groups $(\RR \times \widetilde{\Eg_0(2)})/(\ZZ\gamma_c)$ are isomorphic to $(\RR \times \widetilde{\Eg_0(2)})/(\ZZ\gamma_0)$ and an isomorphism is induced by the linear map $\varphi_c \colon \RR \times \widetilde{\Eg_0(2)}\to \RR \times \widetilde{\Eg_0(2)}$ (seen as vector spaces) defined by the matrix
\[
 \left(
 \begin{matrix}
 1 & 0 & -c/2\pi \\
 0 & E_2 & 0 \\
 0 & 0 & 1 \\
 \end{matrix}
 \right).
\]
A simple calculation shows that $\varphi_c$ is a Lie group isomorphism such that $\varphi(\gamma_c)=\gamma_0$. Thus the group $G$ is isomorphic to $\RR \times \Eg_0(2)$ which cannot be the case since the induced $G$--invariant connection is flat.

Next we consider the central extension of 
\[
0 \to \RR \to \RR \times_{\omega_1} \fe(2)\to \fe(2) \to 0.
\]
There is a basis $(e_0,e_1,e_2,e_3)$ of $\RR \times_{\omega_1} \fe(2)$ such that the Lie brackets are expressed by the relations
\[
 [e_0,e_i]=0,\quad [e_1,e_2]=e_0,\quad [e_1,e_3]=-e_2,\quad [e_2,e_3]=e_1,
\]
for $i=1,2,3$. This means that $e_0$ spans the center and the other the $\fe(2)$ part. Now it is easy to see that $\RR \times_{\omega_1} \fe(2)$ is isomorphic to $\mathfrak{nil}\rtimes_\sigma \RR$ where $\mathfrak{nil}$ is the $3$--dimensional Heisenberg Lie algebra and $\sigma \colon \RR \to \Derg(\mathfrak{nil})$ is defined as $\sigma_{e_3}=\ad_{e_3}$ (see e.g. \cite[p.115]{Konstantis2013}. Thus we obtain the central extension
\[
  0 \lr \RR \lr \mathfrak{nil}\rtimes \RR \lr \fe(2)\lr 0.
\]
where the inclusion is defined by $\lambda \mapsto\lambda e_0$ and the projection is given by $e_0\mapsto 0$ and $e_i\mapsto e_i$ for $i=1,2,3$. Integrating this semidirect product turns out to be a very hard task, therefore we guess the representation $\rho\colon \RR \to \Nil$. Comparing this situation with the $\Nil$--geometry in \cite{Scott1983}  we define (cf. \cite[p. 467]{Scott1983})
\[
 \rho_\theta({\mathbf x},z)=\left(R_\theta({\mathbf x}, z+\tfrac{1}{2}+s(cy^2-cx^2-2sxy)\right)
\]
where ${\mathbf x}=(x,y)$, $R_\theta({\mathbf x})$ rotates ${\mathbf x}$ through $\theta$ in $\RR^2$, $s=\sin\theta$ and $c=\cos\theta$. A straightforward calculation shows that the Lie algebra of $\Nil\rtimes_\rho \RR$ is $\mathfrak{nil}\rtimes_\sigma \RR$ (the computations were done e.g. in \cite[p. 115]{Konstantis2013}). The center is again isomorphic to $\RR \times \ZZ$ and it is embedded in $\Nil \rtimes \RR$ as $(z,l)\mapsto (0,0,z,l)$. The generator of $\pi_1(G)$ has to have the form $\gamma_c=(0,c,2\pi)$. The groups $(\Nil\rtimes\RR)/\ZZ\gamma_c$ are isomorphic to $(\Nil\rtimes\RR)/\ZZ\gamma_0=\Nil\rtimes \SOg(2)$ and an isomorphism is e.g. the linear map $\varphi_c$ given by the matrix
\[
 \left(
 \begin{matrix}
  E_2 & 0 & 0 \\
  0 & 1 &-c/2\pi \\
  0 & 0 & 1 \\
 \end{matrix}
 \right),
\]
where we regard $\Nil\rtimes \RR$ as $\RR^2\times \RR\times \RR$. We conclude $G=\Nil \rtimes \SOg(2)$. Repeating the arguments from the previous proposition, we see that $\widetilde\Eg_0(2)$ is acting like $a.(t,x)=(t+s(x,a),\pi(a).x)$ where $\pi \colon \widetilde{\Eg_0(2)} \to \Eg_0(2)$ is the universal cover homomorphism. But restricting $\Psi$ to $\ker\pi'=\pi_1(G)$ one obtains that $\Psi$ has to be the zero map, since it is either an isomorphism or trivial. Consequently this implies that the isotropy group of $\Nil\rtimes \SOg(2)$ is the $\SOg(2)$--part. Moreover this shows that $M$ is actually the Heisenberg group $\Nil$, which clarifies how $G$ is acting on $M$. 

Finally the other central extensions defined by $\omega_\lambda$ for $\lambda\neq 0,1$ are \emph{weakly isomorphic} to the one defined by $\omega_1$. Here two central extensions $0 \to \fh \to \fg \to \fk \to 0$ and $0 \to \fh \to \fg' \to \fk \to 0$ are called weakly isomorphic if the diagram
\[
  \begin{CD}
     0@>>>\fh@>>>  \fg @>>>  \fk@>>> 0 \\
     @.		   @VVV 		    @VVV		      @VVV\\
      0@>>>\fh@>>>  \fg' @>>> \fk@>>> 0 \\
  \end{CD}.
\]
commutes and the vertical maps are isomorphisms. Clearly two weakly isomorphic extensions induce isomorphic geometries $(M,G)$.

\end{proof} 
\noindent
We sum up the list of geometries with non--flat connection $\omega$ in

\begin{thm}\label{TNonFlatGeometries}
If $(M,G)$ is an axially symmetric geometry such the the $G$--invariant connection is non--flat. Then $(M,G)$ is equivariant diffeomorphic to exactly one the following geometries:
\begin{itemize}
 \item[(a)] $(S^3,\Ug(2))$,
 \item[(b)] $(\widetilde{\SLg}, \Gamma)$,
 \item[(c)] $(\Nil,\Nil\rtimes \SOg(2))$.
\end{itemize}
\end{thm}
\noindent
For a detailed discussion of these geometries see \cite{Scott1983}.

\section{Lie groups as geometries}\label{SLieGroups}

If the stabilizer of $G$ is trivial we may identify $G$ with $M$ through the group action. Since we assumed that $M$ is simply connected we may linearize the problem and classify the $3$-dimensional Lie algebras.

This was done first by Luigi Bianchi in the \nth{19} century, see \cite{Bianchi1898} or \cite{Bianchi2001}. This approach was translated in \cite{Glas2013} to a more modern, coordinate free language. Here the Lie algebras are divided into two classes: in the unimodular ones and the non-unimodular. Geometrically speaking this means that if $G$ admits a compact quotient, then $G$ has to be unimodular, compare \cite[p. 99]{Milnor1976}. Thus the non--unimodular Lie groups cannot posses cocompact subgroups and therefore they do not appear in Thurston's list.

Another way to classify the $3$-dimensional Lie algebras is to consider the derived Lie algebra $\fg'=[\fg,\fg]$ which isomorphism class is an invariant for the isomorphism class of $\fg$. In particular the dimension of $\fg'$ does not change under Lie algebra isomorphisms. This approach may be found in full detail \cite[p. 67]{Konstantis2013} OR in \cite{Konstantis}. We would like to describe briefly this approach and the resulting $3$-dimensional Lie algebras.

\bigskip
\noindent
{$\bullet$ ${\dim \fg=0}$}\;: Then $\fg$ is abelian and $\fg =\RR^3$.

\bigskip
\noindent
{$\bullet$ ${\dim \fg=1}$}\;: We obtain the short exact sequence of Lie algebras
\[
 0 \longrightarrow \RR \longrightarrow \fg \longrightarrow \RR^2 \longrightarrow 0.
\]
The resulting Lie algebras are $\fh^2 \times \RR$ where $\fh^2$ is the non--abelian $2$-dimensional Lie algebra and the Lie algebra of the Heisenberg group.

\bigskip
\noindent
{$\bullet$ ${\dim \fg=2}$}\;: We obtain the short exact sequence of Lie algebras
\[
 0 \longrightarrow \RR^2 \longrightarrow \fg \longrightarrow \RR \longrightarrow 0.
\]
Obviously this sequence always splits. Suppose $e_3$ is a non-zero vector in $\RR$. A splitting map determines the Lie brackets of $e_3$ with an element of $\fg'=\RR^2$. Since all other Lie brackets are trivial, the linear map $\ad_{e_3} \colon \fg' \to \fg'$ determines the Lie algebra. The remaining task is to determine what linear maps $\ad_{e_3}$ produce isomorphic Lie algebras, which is an easy exercise (compare \cite[p. 69]{Konstantis2013}). There are $3$ types of Lie algebras and two of them come with a continuous family of non-isomorphic Lie algebras.

\section{The complete list of geometries}

In this last section we would like to write down the final list of $3$-dimensional geometries.

\begin{thm}\label{TBigTheorem}
Let $(M,G)$ be a geometry. Then $(M,G)$ is equivariant diffeomorphic to exactly one of the following geometries:
\begin{itemize}
 \item[(i)] {isotropic geometries} ($\dim K=3$).
 \begin{center}
 $(\RR^3, \Eg_0(3))$\;
 $(S^3, \SOg(4)$\;
 $(D^3, \Hg_0(3))$\;
 \end{center}
\item[(ii)] {axially symmetric geometries} ($\dim K=1$).
\begin{center}
 \begin{tabular}[ht]{c|c}
  flat & non--flat \\
  \hline\hline
  $(S^2\times \RR, \SOg(3)\times \RR)$ & $(S^3, \Ug(2))$\\
  $(D^2\times \RR, \SOg^+(2,1)\times \RR)$ &$(\widetilde{\SLg}, \Gamma)$\\
  $(\RR^2\times \RR, \Eg_0(2)\rtimes_{\rho_i} \RR)$, $i=0,1$ & $(\Nil,\Nil\rtimes \SOg(2))$\\
\end{tabular}
\end{center}

\item[(iii)] Lie groups in dimension $3$ ($\dim K=0$), compare the list in \cite[p. 73]{Konstantis2013}

\end{itemize}
\end{thm}

\bibliography{NCTG}{}
\bibliographystyle{geralpha}
\end{document}